\documentclass[12pt]{amsart}

\usepackage[T1]{fontenc}
\usepackage[utf8]{inputenc}
\usepackage{amssymb}
\usepackage{color}
\newcommand{\cal}{\mathcal}

\usepackage{amsthm}
\theoremstyle{plain}
\newtheorem{lemma}{Lemma}[section]%[lemma]   REMETTRE POUR NUMEROTER SEREINEMENT
\newtheorem{thm}[lemma]{Theorem}
\newtheorem{cor}[lemma]{Corollary}
\newtheorem{prop}[lemma]{Proposition}
\theoremstyle{definition}   %[lemma]
\newtheorem{definition}[lemma]{Definition}
\newtheorem{rem}[lemma]{Remark}
\newtheorem{example}[lemma]{Example}

\renewcommand{\vec}{\overrightarrow}
\newcommand{\eps}{\varepsilon}
\DeclareMathOperator{\Isom}{Isom}
	
\begin{document}

\title[Tangent bundles and $\ell^p$]{Tangent bundles of hyperbolic spaces and proper affine actions on $L^p$ spaces}
\author [Chatterji, Dahmani, Haettel and L\'ecureux]{Indira Chatterji, Fran\c cois Dahmani, Thomas Haettel and Jean L\'ecureux}
%\date{\today}

%%%%
%%%%   A FAIRE
%%%%
%%%%   
%%%%   
%%%%   (relire des preuves)
%%%%

\maketitle
\section*{Introduction}
Let $G$ be a locally compact, second countable  group, and $V$ be a normed vector space. The group of affine isometries of an affine space of linear part $V$, is isomorphic to $V\rtimes \vec{{\rm Isom}} V$, where $\vec{{\rm Isom}} V$ is the group of linear isometries of $V$. Thus, an isometric affine action of $G$ on $V$ gives, by quotient, a representation $\lambda: G\to \vec{{\rm Isom}} V$, and, by cancellation with a section, a map $c:G\to V$ satisfying the {\it cocycle condition} for $\lambda$, namely 
$$c(gh) = c(g) + \lambda(g) c(h)$$ 
for all $g,h\in G$. Conversely, a representation  $\lambda: G\to \vec{{\rm Isom}} V$ and a cocycle $c: G\to V$ produce an affine isometric action of $G$ on $V$ (seen as affine space) by the formula 
$$gv= \lambda(g)v+ c(g).$$ 
The cocycle $c$ is a {\it coboundary}, meaning that it satisfies $c(g)=d-\lambda(g)d$ (for some $d\in V$ and any $g\in G$) if and only if there is  a fixed vector in $V$ for $G$ (in which case, it is that vector $d$).  The cocycle is proper (meaning that $\|c(g_n)\| \to \infty$ if $g_n\to\infty$ in the sense that it leaves every compact subset of  $G$) if and only if the action is metrically proper (meaning that for all $v$, $\|g_nv\|\to \infty$ if $g_n\to\infty$). 

For $V$ a Hilbert space, the existence of a proper isometric action of $G$ on $V$ is the {\it Haagerup property} for $G$,  and, if $G$ is not compact,  is a fierce negation of {\it property (T)}, since the latter is equivalent to the existence of a fixed point for any action on a Hilbert space. 

In this note we are interested in the case where $V$ is an $L^p$-space for $p>1$. A proper action of $G$ on an $L^p$-space is a weakening of a too strong rigidity property of the type of property (T). A contrario, according to \cite{BFG}, the lattices $SL_n(\mathbb{Z})$, for $n\geq 3$, have a fixed point property for all actions on $L^p$-spaces (for $1<p<\infty$). 

Discrete countable groups acting properly discontinuously on CAT(0) cubical complexes do act properly on a Hilbert (or $L^2$) space \cite{NR97, NR98}. All hyperbolic groups do act properly on some $L^p$-space, for some $p>1$ \cite{Yu, AL, Ni}, whereas there are hyperbolic groups with property (T) (for instance lattices in $Sp(n,1)$ according to Kostant \cite{Kos}). Fisher-Margulis according to \cite{BFG} show that if a group has property (T), then there is an $\epsilon>0$ such that this group has a fixed point for any action on an $L^p$ and for all $p\in[2,2+\epsilon)$.

Let us illustrate these results in the case of $G$ a free group. Let $\mathcal{E}$ be the set of oriented edges of a tree on which $G$ acts freely and let $V_1=\ell^2(\mathcal{E}, \mathbb{R})$. The representation $\lambda$  is induced by the action of $G$ on the tree. Let $d$ be the vector of $\ell^\infty (\mathcal{E}, \mathbb{R})$ defined by $d(e)=1$ if $e$ points toward $1$ in the tree (meaning that the end of $e$ is closer to $1$ than the origin of $e$), and $d(e)=0$ otherwise (note that $d$ is not in $V_1$ as it takes value 1 on an infinite set). Then $c(g)=d-\lambda(g)d$ is a map a priori from $G$ to  $\ell^\infty (\mathcal{E}, \mathbb{R})$. It is a coboundary, hence a cocycle.  But the support of $c(g)$ is the set of edges in the oriented segments  $[1,g]$  and  $[g,1]$. These supports are finite, and therefore $c$ takes its values in $V_1$. One checks that $\|c(g)\|$ grows like the square root (because we are in $\ell^2$) of the length of the interval $[1, g]$ in the tree, hence the cocycle $c$ is proper.  

Now, let $V_2= \ell^2(G, \ell^1(G,\mathbb{R}))\subseteq\ell^2(G\times G,\mathbb{R})$. If the tree on which $G$ acts is the Cayley tree, the representation $\lambda$ is the natural one, extending to the diagonal action of $G$ on $G\times G$. We define a vector $d\in\ell^\infty (G, \ell^1(G,\mathbb{R}))$
 as the map that, to $h\neq 1$ associates the Dirac mass on the neighbor of $h$ closest to $1$, and $d(1)$ is the null function. One notices that under the natural embedding of $\mathcal{E}\subseteq G\times G$, this $d$ extends by 0 the vector constructed above. Again, $d\in\ell^\infty (G, \ell^1(G,\mathbb{R}))$, but $d$ is not in $V_2$. Defining $c(g)=d-\lambda (g)d$, one gets a cocycle from $G$ to a priori $\ell^\infty (G, \ell^1(G,\mathbb{R}))$, but that actually takes values in $V_2$ since $c(g)$ has finite support for all $g$. Indeed, for a fixed $h\in G$, the map $c(g)(h):G\to\mathbb{R}$ is null everywhere except possibly on the two neighbors of $h$ that are respectively closer to $1$ and to $g$, on which it takes values $1$ and $-1$, or $0$ if these points are equal. Hence for $g$ fixed, the map $c(g)(\cdot):G\to\mathbb{R}$ is finitely supported on the interval $[1,g]$ in the tree and has $\ell^2$ norm proportional to the distance between 1 and $g$, showing that $c$ is in fact a proper cocycle.

This idea has been generalized using a coarse version of "taking the neighbor closest to $1$". The constructions involved in \cite{Yu} and in \cite{AL} use a coarse geodesic flow from $1$ to $h$ and from $g$ to $h$ and the comparison of the arrivals of these flows. In free groups, they arrive exactly at the same point if $h$ is not on the segment $[1,g]$, but in hyperbolic groups, there could be a difference. Still, this difference is exponentially small in the Gromov product $(1|\,g)_h$, allowing to adjust the exponent $p$ to beat the growth of the group. 

In all these cases, one has to evaluate "how much" an element $h$ should be thought of as being between $1$ and $g$. In both cases, one estimates the difference between how one does see $1$  from $h$ and how one does see $g$ from $h$. 

In this note, we cast this point of view in the language of tangent bundles. We define a suitable notion of a \emph{tangent space} for a metric measured space, endowed with  a group action,  and give a definition of negative curvature for such a tangent space. We then prove that properties of this  tangent space give rise to an action of our group on some $L^p$ space, which is proper under some suitable assumptions.

\bigskip

A Radon measure $\mu$ on a locally compact geodesic hyperbolic space $X$ is \emph{non-collapsing} if there exists two constants $C \geq 0$, and   $ v>0$, such that for all  $ x \in X$, one has that $\mu(B(x,C)) \geq v$. For instance, the counting measure on a locally finite graph is non-collapsing,  and the volume form on a simply connected complete Riemannian manifold with sectional curvature $\leq -1$ is non-collapsing,  by Gunther's inequality (see \cite[Theorem 3.101 (ii)]{GHL}).

Recall that the volume entropy of such a measure $\mu$ is the exponential growth rate of the $\mu$-measure of balls in $X$ (see Definition \ref{def;vol_entropy}).
\begin{thm}\label{main} Let $\delta\geq 0$ and $X$ be a locally compact geodesic $\delta$-hyperbolic space, and let $\mu$ denote an non-collapsing Radon measure on $X$. Let $h \geq 0$ denote the volume entropy of $\mu$. Assume that a second countable locally compact group $G$ acts properly by isometries on $X$,  preserving $\mu$. Then for any $p>\max(1,h\delta/\log(2))$, $G$ acts properly by affine isometries on a $L^p$-space.
\end{thm}
One cannot hope to strengthen this result to encompass acylindrical actions: indeed Minasyan and Osin in  \cite{MO} proved that some acylindrically hyperbolic groups have a fixed point property for all $L^p$-spaces. The following corollary applies for instance to a geometrically finite group action. Observe that these groups are relatively hyperbolic with virtually nilpotent parabolic subgroups.

Recall that a Hadamard manifold is a complete simply connected Riemannian manifold, with non-positive sectional curvature. Its curvature is $\alpha$-pinched negative, where $0 < \alpha \leq 1$, if it takes value in an interval $[\kappa, \alpha \kappa]$ with $ -\infty < \kappa \leq \alpha \kappa<0$. 

\begin{cor}\label{propre} Let $M$ be a $\alpha$-pinched negative curvature Ha\-da\-mard manifold of dimension $n$, and let $G$ be a locally compact second countable group acting properly discontinuously and by isometries on $M$. Then, for any $p > \frac{n-1}{\sqrt{\alpha}}$,  $G$ acts properly on an $L^p$-space.
\end{cor}
Our notion of tangent bundle also gives an alternate proof of the following result.
\begin{cor}[Cornulier-Tessera-Valette \cite{CTV}]\label{CTV}
	Let $G$ be a simple algebraic group of rank $1$ over ${\mathbb{R}}$ or ${\mathbb{C}}$, and let $\delta_H$ denote the Hausdorff dimension of the visual boundary of the symmetric space $X$ of $G$. Then for any $p > \max(1,\delta_H)$, $G$ has a proper affine action on some $L^p$-space.
\end{cor}
{\it Acknowledgments.} We are grateful to Mika\"el de la Salle for useful discussions, and to the anonymous referee for their comments improving this note.
\section{Tangent bundle on a metric space}\label{tangent}
In this section we define a tangent bundle in the general setting of  a metric measured space.
% , although  we will mainly remain in the context of discrete metric spaces, where the measure is the counting measure.
%
\begin{definition}
Let $(X,d,\mu)$ be a locally compact metric measured space, where $\mu$ is a Radon measure. We say that $X$ has a {\it tangent bundle} $TX$ if:
\begin{enumerate}
\item $TX$ is a Polish space, with a Borel map $\pi:TX\to X$
\item for every $a\in X$, the fiber $\pi^{-1}(a)$ is a Banach space, denoted $T_aX$
\item there is a measurable map $X\times X\to TX$,  $(a,x)\mapsto \vec{ax}$, so that $\vec{ax}\in T_aX$ for all $a$ and all $x$, and so that $\vec{aa} = 0$ for all $a$. 
\end{enumerate}
 For $\kappa\leq 0$, we say that the tangent bundle has {\it curvature at most $\kappa$} if for every $C \geq 0$, there exists $D_C \geq 0$ such that, for each $a,x,y \in X$ with $d(x,y) \leq C$, we have 
$$\|\vec{ax}-\vec{ay}\| \leq D_C e^{\kappa d(a,x)}.$$
We say that the tangent bundle  is {\it proper} if there exists a proper function $f:[0,+\infty)\mapsto [0,+\infty)$ such that, for every $x,y \in X$, for every $p >1$, we have 
$$\int_{a \in X} \|\vec{ax}-\vec{ay}\|^p d\mu(a) \geq f(d(x,y)).$$
If a group $G$ acts by measure preserving isometries on $X$, we say that $TX$ is {\it $G$-equivariant} if, for every $a$ and $g \in G$, there is an isometry $\phi_g:T_aX\to T_{g  a}X$ such that for every $x \in X$, we have $$\phi_g(\vec{ax}) = \vec{(g  a) (g x)}.$$
\end{definition}
\begin{example}\label{ex;1.2}
Let $M$ be a uniquely geodesic Riemannian manifold, and $TM$ its usual tangent bundle. For $a\neq x$ in $M$, define the vector $\vec{ax}$ in $T_aM$ to be the unit vector tangent to the geodesic from $a$ to $x$. This data endows $M$ with a tangent bundle in the sense above.  If the sectional curvature of $M$ is at most $\kappa$, then the tangent bundle has curvature at most $\kappa$.  Note that contrary to what the notation may suggest, in this example,  the vector has norm $1$, regardless of the distance between $a$ and $x$. That tangent bundle is $\Isom(X)$-equivariant.
\end{example}
\begin{example}
Let $Y$ be a countable simplicial tree, with the graph metric and the counting measure. For every vertex $v$, we set $T_vY$  to be $\ell^2({\rm Lk}(v),\mathbb{R})$, where ${\rm Lk}(v)$ is the set of vertices adjacent to $v$. Let $TY$ be the disjoint union of the spaces $T_aY$. For $a \neq x$ in $Y$, let $\vec{ax}$ be the indicator of the neighbor of $a$ closest to $x$.  This endows $Y$  with a  ${\rm Isom}(Y)$-equivariant  proper tangent bundle with  curvature at most $\kappa$, for all $\kappa$.  
\end{example}
\begin{example}[Alvarez-Lafforgue \cite{AL}] \label{ex;AlvarezLafforgue} Let $X$ be a hyperbolic graph, uniformly locally finite (that is, all balls of fixed radius have a uniform bound on their cardinality), with its graph metric and counting measure. Let $\delta$ be a hyperbolicity constant, and for all $a\in X$, let $T_aX$ be the Hilbert (euclidean) space of maps from the ball $B(a,4\delta)$ of radius $4\delta$ around $a$, to $\mathbb{R}$. Let $TX$ be the disjoint union of the $T_aX$, and for each $x\neq a$, let $\vec{ax}\in T_aX$ be the map $\mu_x(a): B(a, 4\delta)\to \mathbb{R}$ constructed by Alvarez and Lafforgue in \cite[Th\'eor\`eme 4.1]{AL}.   This endows $X$ with an ${\rm Isom}(X)$-equivariant proper tangent bundle with negative curvature. In this construction, the vector $\vec{ax}$ has unit norm for the $\ell^1$-norm on $T_aM$.
\end{example}	
According to \cite{CD} this result remains true if one relaxes the assumptions to the case of  hyperbolic graphs that are possibly not locally finite, but that are {\it uniformly fine}, allowing coned-off graphs of relatively hyperbolic groups. We will explain how hyperbolic spaces with controlled growth admit proper tangent bundles of negative curvature in Section \ref{sec;examplesHyp}.

\begin{definition}\label{puniform}
Let $1\leq p\leq \infty$. The tangent bundle $TX$ is said to be \emph{$p$-uniform} if there exists a measure space $\Omega$ and a Borel isomorphism between $TX$ and $X\times L^p(\Omega)$, such that $\pi$ corresponds by this isomorphism to the first projection.
\end{definition}
\begin{definition}
Let $(X,d,\mu)$ be a locally compact metric measured space, and assume that $X$ has a tangent bundle $TX$. For $1\leq p\leq\infty$, we define $L^p(X,TX)$ as the set of measurable functions $f: X\to TX$ such that $f(x) \in T_xX$ for every $x\in X$  and  
$$\| f\|_p:=\left(\int_X \|f(x)\|^p d\mu(x)\right)^{\frac{1}{p}} < \infty.$$
\end{definition}
\begin{rem}\label{rem;Fubini}
  If $TX$ is $p$-uniform, then there exists a measure space $Z$ such that $L^p(X,TX)=L^p(Z)$. Indeed in that case we get, up to
   isometry, % Borel isomorphism,
  that for all $x$,  $T_xX= L^p(\Omega)$, and therefore, by the previous definition,  $L^p(X,TX)=L^p(X,L^p(\Omega))=L^p(X\times \Omega)$, the last equality being Fubini's theorem.
\end{rem}
\section{Actions on $L^p$-spaces from negatively curved tangent bundles}
We now to explain how tangent bundles of negative curvature relate to actions on $L^p$-spaces. 
\begin{definition}\label{def;vol_entropy}
Let  $(X,d,\mu)$ be a locally compact metric measured space. The \emph{volume entropy} of $X$ is the number $h_{vol}$ defined as the limit 
$$h_{vol}=\limsup\limits_{r \rightarrow \infty} \frac{\log \mu(B(x,r))}{r}.$$
\end{definition}
The volume entropy is independent of $x$ since eventually, as $r$ grows the ball $B(x,r)$ will contain any other point of $X$.
\begin{thm}\label{thm:properaction}
Assume that a second countable locally compact group $G$ acts on a locally compact metric measured space $(X,d,\mu)$ of finite volume entropy by measure-preserving isometries. If $X$ has a $G$-equivariant tangent bundle with curvature at most $ \kappa<0$, then for any $p>\frac{h}{|\kappa|}$, the group $G$ admits an affine action on $L^p(X,TX)$, which is proper when the tangent bundle is proper.
\end{thm}
\begin{proof}
For $p >\frac{h}{|\kappa|}$, let $V=L^p(X,TX)$. We define a linear isometric action $\pi$ of $G$ on $V$ by
$$\pi_g(f)(x)= \phi_g(f(g^{-1} x)).$$%
Note that $\pi_g(f)(x) \in \phi_g( T_{g^{-1}x}X ) = T_xX$ for any $g \in G$ and $f \in V$. We then fix a basepoint $o \in X$, and set $f_o :X \mapsto TX, x\mapsto \vec{xo} \in T_xX$, so that we can define a cocyle by
\begin{align*} c : G &\to \ell^{\infty}(X,TX) \\
g & \mapsto f_o - \pi_g (f_o).\end{align*}
Indeed, since $\pi_g(f_o)(x)\in T_xX$, the map $c$ is well-defined, and it is a cocycle since it is a coboundary. We now prove that $c$ is integrable, i.e. takes values in $V$. Fix a $g \in G$ and let $C=d(o,go)$. We have
\begin{align*} \|c(g)\|_V^p &= \int_{x \in X} \|c(g)(x)\|^p_{T_xX} d\mu(x) = \int_{x \in X} \|\vec{ x o} - \vec{x go}\|^p_{T_xX} d\mu(x) \\
& \leq  \int_{x \in X} D_C^p e^{p\kappa d(x,o)} d\mu(x) \leq  \int_0^{\infty} D_C^p e^{p\kappa r} e^{h r} d r < \infty, \end{align*} 
since $p\kappa + h < 0$.

If we assume furthermore that $TX$ is a proper tangent bundle, we have $\|c(g)\|^p \geq f(d(o,go))$ for some proper function $f:[0,\infty)\to [0,\infty)$. Since the action of $G$ on $X$ is proper, we deduce that $c$ is proper as well.
\end{proof}
\section{Tangent bundles on hyperbolic spaces}\label{sec;examplesHyp}
The main result of this section is the construction of a tangent bundle for a $\delta$-hyperbolic space under some mild assumptions.
\begin{prop} \label{prop:TX}
Let $(X,d)$ be a locally compact geodesic $\delta$-hyperbolic space, and some $0<\eps<\frac{\log(2)}{\delta}$. Consider a non-collapsing Radon measure $\mu$ on $X$, with finite volume entropy $h$. Let $G$ denote a group acting by isometries on $X$ and preserving $\mu$. Then $X$ has a proper, $G$-equivariant tangent bundle $TX$ with curvature at most $-\eps$. Furthermore $TX$ can be chosen to be $p$-uniform (see Definition \ref{puniform}) for every $1\leq p<\infty$.
\end{prop}
In order to construct the tangent bundle we need an adaptation of \cite[Proposition 7.10]{GhysHarpe}, suggested in \cite{AL2}, that we now describe. Similar constructions, to define visual metrics on the boundary, can be found also in \cite[Proposition 3.21]{BH} and \cite[Lemme 1.7]{CDP}, but we want to define a metric in the space instead of the boundary, so we adapted the construction. Recall that in a metric space $(X,d)$, for a base point $a\in X$ and two points $x,y\in X$, the Gromov product is defined by
$$(x|y)_a={\frac{1}{2}}(d(x,a)+d(y,a)-d(x,y))$$
and that $X$ is $\delta$-hyperbolic if and only if, for any points $x,y,z\in X$, one has that
$$(x|y)_a\geq \min\{(x|z)_a,(y|z)_a\}-\delta.$$
\begin{prop} \label{pro:angle}
Fix a $\delta$-hyperbolic geodesic space $X$, and choose $\eps,D > 0$ such that $0<\eps \leq \frac{\log(2)}{\delta+D}$. Let $\alpha=De^{-2D\eps}$ and $\beta=\frac{8}\eps$. For any $a \in X$, there exists a pseudo-distance $d^a_\eps$ on $X$, invariant under the isometry group of $X$, such that 
\begin{enumerate}
\item $d^a_\eps (x,y) \leq \beta e^{-\eps (x|y)_a}$ for all $x,y \in X$,  
\item $\alpha e^{-\eps (x|y)_a} \leq d^a_\eps (x,y) $ for every $x,y\in X$ with $d(x,y)\geq 2D$.
\end{enumerate}\end{prop}
%We copy here \cite[Proposition 7.10]{GhysHarpe} and adapt it to define a distance function on $X$ instead of the boundary $\partial X$. 
\begin{proof} Fix $a \in X$, and define, for $x,y \in X$,
$$d^a_\eps(x,y) = \inf_{c,L} \left\{\int_0^L e^{-\eps d(a,c(t))} dt\right\},$$ 
where the infimum is taken over all $L \geq 0$ and $1$-Lipschitz maps $c:[0,L] \rightarrow X$ such that $c(0)=x$ and $c(L)=y$.
From the definition, $d^a_\eps$ is symmetric, nonnegative and satisfies the triangle inequality. It may not be definite though, hence $d^a_\eps$ is merely a pseudo-distance on $X$.

\medskip

(1) We first look at the upper bound for $d^a_\eps$. Fix $x,y \in X$ and a geodesic $c:[0,L] \rightarrow X$ from $x$ to $y$. Let $T \in [0,L]$ such that $|d(a,c(T)) - (x|y)_a| \leq \delta$. For every $t \in [0,T]$, we have $|d(c(t),a) - (T-t) - d(a,c(T))| \leq \delta$. Therefore
$$ \int_0^T\!\!\! e^{-\eps d(c(t),a)}dt \leq e^{\eps \delta} e^{-\eps d(a,c(T))} \int_0^T\!\!\! e^{-\eps(T-t)} dt \leq \frac 1\eps e^{2\eps \delta} e^{-\eps (x|y)_a}\leq \frac 4\eps e^{-\eps (x|y)_a}$$
since $e^{\eps \delta} \leq 2$. Similarly $\int_T^L e^{-\eps d(c(t),a)}dt \leq \frac 1\eps e^{2\eps \delta} e^{-\eps (x|y)_a}$ and we have proved that $d^a_\eps(x,y) \leq \beta e^{-\eps (x|y)_a}$.

\medskip

(2) Let us now look at the lower bound, that we obtain by induction on $L \geq D$. Precisely, fix $x,y \in X$, with $d(x,y)\geq 2D$, and fix a $1$-Lipschitz path $c:[0,L] \rightarrow X$ from $x$ to $y$, hence $L\geq D$. Let ${\cal L}$ be the set of all $L' \in [D,L]$ such that, for any interval $[u,v] \subseteq [0,L]$ with $D\leq v-u\leq L'$, we have $De^{-2D \eps} e^{-\eps(c(u)|c(v))_a} < \int_u^v e^{-\eps d(a,c(t))} dt$. We will prove that the interval ${\cal L}$ is non-empty and both open and closed in $[D,L]$, so that it has to be the whole $[D,L]$, for any $L \geq D$.

\medskip

We first show that ${\cal L}$ is non-empty, by showing that $[D,2D) \subseteq {\cal L}$. When $D\leq  L' < 2D$ and $[u,v]$ has length $L'$, then for all $t \in [u,v]$ we have by triangular inequality 
\begin{align*}
2d(a,c(t))&\leq d(a,c(u))+d(a,c(v))+d(c(u),c(t))+d(c(v),c(t))\\
&\leq 2(c(u)|c(v))_a+d(c(u),c(v))+v-t+t-u\\
&< 2(c(u)|c(v))_a+4D.
\end{align*}
Hence
$$\int_u^v e^{-\eps d(a,c(t))} dt > L' e^{-2D\eps} e^{-\eps (c(u)|c(v))_a} \geq De^{-2D\eps} e^{-\eps (c(u)|c(v))_a}.$$
So $L' \in {\cal L}$, and $[D,2D) \subseteq {\cal L}$.%\footnote{I don't see why. We might have a problem when $u$ and $v$ are close ?  Thomas: Yes ! When u,v are close the "cutting in the midpoint" process described after might fail}

\medskip

That ${\cal L}$ is open in $[D,L]$ is because the condition defining ${\cal L}$ is open and the interval $[D,L]$ is compact.

\medskip

We will now show that ${\cal L}$ is closed in $[D,L]$. Fix $L' \in (D,L]$, and assume that $[D,L') \subseteq {\cal L}$, we will prove that $L' \in {\cal L}$. Fix an interval $[u,v] \subseteq [D,L]$ of length $L'$, and let $R=\int_u^v e^{-\eps d(a,c(t))} dt$. 
Let $w \in [u,v]$ such that $\int_u^w e^{-\eps d(a,c(t))} dt=\int_w^v e^{-\eps d(a,c(t))} dt=\frac{R}{2}$. We know that the intervals $[u,w]$ and $[w,v]$ both have length shorter than $L'$. If both intervals $[u,w]$ and $[w,v]$ have length greater than $D$, the induction hypothesis say that $De^{-{2D\eps}} e^{-\eps(c(u)|c(w))_a} < \frac{R}{2}$ and $De^{-{2D\eps}} e^{-\eps(c(w)|c(v))_a} < \frac{R}{2}$. 
By hyperbolicity, $e^{-\eps(c(u)|c(v))_a}\leq 2 \max\{e^{-\eps(c(u)|c(w))_a},e^{-\eps(c(w)|c(v))_a}\}$, hence we conclude that
$$De^{-{2D\eps}} e^{-\eps(c(u)|c(v))_a} < R = \int_u^v e^{-\eps d(a,c(t))} dt.$$
Otherwise, one interval, say $[u,w]$, has length smaller than $D$, and so $d(c(u),c(w))<D$ since $c$ is $1$-Lipschitz. As $(c(u)|c(v))_a \geq (c(w)|c(v))_a -d(c(u),c(w)) > (c(w)|c(v))_a - D$, we deduce that
$$e^{-\eps (c(u)|c(v))_a} < e^{D \eps} e^{-\eps (c(w)|c(v))_a}.$$
Since the length of $[u,w]$ is smaller than $D$ and $[u,v]$ has length at least $2D$, we deduce that $[w,v]$ has length at least $D$. By the induction hypothesis we have $De^{-{2D\eps}} e^{-\eps(c(w)|c(v))_a} < \frac{R}{2}$. Thus
$$De^{-{2D\eps}} e^{-\eps(c(u)|c(v))_a} < e^{D \eps} \frac{R}{2} \leq R=\int_u^v e^{-\eps d(a,c(t))} dt,$$
since $\frac{e^{D\eps}}{2} \leq 1$, as we have assumed that $(D+\delta)\eps \leq \log(2)$. So we have proved that $L' \in {\cal L}$, and hence ${\cal L}$ is closed in $[D,L]$.

\medskip

In conclusion, we have $De^{-2D\eps }e^{-\eps(x|y)_a} \leq d_\eps(x,y)$ for all $x,y \in X$ with $d(x,y)\geq 2D$.
\end{proof}
We can now prove Proposition \ref{prop:TX}. The construction of the tangent bundle uses the pseudo-distances constructed above, and the desired properties rely on the estimates for those pseudo-distances.
\begin{proof}[Proof of Proposition \ref{prop:TX}]
Let $1\leq p<\infty$. For any $a \in X$, we set $T_aX=L^p(X,{\mathbb{R}},\mu)$ and define 
$$TX=\bigcup_{a\in X} T_aX=X\times L^p(X,{\mathbb{R}},\mu),$$ 
which we equip with the product topology. Fix $o\in X$, for each $a,x \in X$, define $\vec{ax}\in L^{\infty}(X,{\mathbb{R}},\mu)$ by
$$\vec{ax}(\xi)= d^{a}_\eps(x,\xi)\frac{e^{-d(a,\xi)}}{f(d(a,\xi))^{1/p}}$$
where $d^a_\eps$ is the pseudo-distance defined in Proposition~\ref{pro:angle} and 
$$f: {\mathbb{R}}_+\to {\mathbb{R}}_+,\ r\mapsto \mu(B(o,r))$$ 
(which is proper because $\mu$ is non-collapsing). 

\medskip

We first notice that $\vec{ax} \in T_aX$. Indeed, since $\mu$ has finite volume entropy, there exists a constant $h' \geq 1$ such that $f(n+1) \leq h' f(n)$ for all $n \in \mathbb{N}$ and we compute, using that $d^{a}_\eps\leq\beta$ from Proposition \ref{pro:angle}(1) 
\begin{align*}\| \vec{ax}\|^p &= \int_{y \in X} \vec{ax}(y)^p d\mu(y) \leq\int_{y \in X} \beta^p \frac{e^{-pd(a,\xi)}}{f(d(a,\xi))}d\mu(y) \\
&\leq \beta^p\sum_{n=0}^\infty  \frac{e^{-pn}}{f(n)} (f(n+1)-f(n)) \leq \beta^p \sum_{n=0}^\infty (h'-1)e^{-pn} = \beta^p E  <\infty,\end{align*}
where $E=\sum_{n=0}^\infty (h'-1)e^{-pn}$. This gives $TX$ the structure of a tangent bundle to $X$.  

\medskip

We now show that the curvature of this bundle is at most $-\eps$. For any $a,x,y,z \in X$, we have 
\begin{align*}
|\vec{ax}(z)-\vec{ay}(z)| &= \left|d_\eps^{a}(x,z) - d_\eps^{a}(y,z) \right| \frac{e^{-d(a,z)}}{f(d(a,z))^{1/p}}\\ & \leq d_\eps^{a}(x,y)\frac{e^{-d(a,z)}}{f(d(a,z))^{1/p}}\end{align*}
hence $\|\vec{ax}-\vec{ay}\|^p \leq E d_\eps^{a}(x,y)^p$. Now, according to Proposition \ref{pro:angle}(1), there exists a constant $\beta > 0$ such that $d_\eps^{a}(x,y) \leq \beta e^{-\eps (x|y)_a}$. So, for any $C\geq 0$ and any $x,y \in X$ such that $d(x,y) \leq C$, we have 
$$\|\vec{ax}-\vec{ay}\| \leq  E^{1/p} \beta e^{-\eps (x|y)_a} \leq D_C e^{-\eps d(a,x)}$$
because $(x|y)_a\geq d(x,a)-C$, where $D_C$ is any constant larger than $E^{1/p} \beta e^{\eps C}$. That is, the curvature of $TX$ is at most $-\eps$.

\medskip

It remains to show that the tangent bundle is proper. We choose $C$ is sufficiently large so that $\alpha - \beta e^{-\eps C}>0$, where $\alpha$ and $\beta$ are the constants of Proposition \ref{pro:angle} and let $C'=2D+\frac{5C}{2}>0$, where $D$ is again as in Proposition \ref{pro:angle}.
For $x,y \in X$, we define 
$$A(x,y)=\{a \in X | d(x,a)+d(a,y) \leq d(x,y)+C, d(a,x), d(a,y) \geq C'\},$$ 
this is the set of points that are almost on a geodesic between $x$ and $y$, and far enough from $x$ and $y$ (it could be empty is $x$ and $y$ are too close, but we will only need the case when $x$ and $y$ are far apart). For any $x,y \in X$, assuming that $z\in A(x,a)$, then
 \begin{align*}
 (x|z)_a &= \frac{d(a,x)+C-C+d(a,z)-d(x,z)}{2}\\ & \geq\frac{d(a,z)+d(z,x)-C+d(a,z)-d(x,z)}{2}\\
 &=d(a,z)-\frac C2\geq C'-\frac{C}{2}=2D+2C,
 \end{align*}
% Note that, when $z \in A(x,y)$ and $a \in A_C(x,z)$, then
% \begin{align*} (x|z)_a &= \frac{d(a,x)+d(a,z)-d(x,z)}{2} \\
% &\geq \frac{d(a,z)+d(z,x)+C+d(a,z)-d(x,z)}{2} \\
% &\geq d(a,z) - \frac{C}{2} \geq C'-\frac{C}{2}=2D+2C,\end{align*}
so that, with Proposition \ref{pro:angle}(1), we obtain $d_\eps^{a}(x,z) \leq \beta e^{-\eps (2D+2C)}$. Similarly, again assuming that $z\in A(x,a)$, we have that
\begin{align*} (y|z)_a &= \frac{d(a,y)+d(a,z)-d(y,z)}{2} \\ &\leq \frac{d(a,y)+d(a,x)-d(x,z)+C-d(y,z)}{2} \\
&\leq \frac{d(a,y)+d(a,x)-d(x,y)+C}{2} \leq C.\end{align*}
If we furthermore assume that $a\in A(x,y)$, we then have
\begin{align*} d(y,z) &\geq d(x,y)-d(x,z) \geq d(x,y) - d(x,a)+d(a,z)-C\\
&\geq d(x,y) - d(x,a)-C \geq d(a,y)-2C \geq C'-2C \geq 2D.\end{align*}
So with Proposition \ref{pro:angle}(2), we deduce that $d_\eps^{a}(y,z) \geq \alpha e^{-\eps C}$. Since we get the same estimate when $z\in A(a,y)$, for any $x,y \in X$ and $a \in A(x,y)$, we have
\begin{align*}
\|\vec{ax} - \vec{ay}\|^p & = \int_{z \in X} |d_\eps^{a}(x,z) - d_\eps^{a}(z,y)|^p \frac{e^{-pd(a,z)}}{{f(d(a,z))}} d\mu(z) \\
 & \geq\int_{z \in A(x,a) \cup A(a,y)} |d_\eps^{a}(x,z) - d_\eps^{a}(z,y)|^p \frac{e^{-p d(a,z)}}{{f(d(a,z))}} d\mu(z)\\
 & \geq \int_{z \in A(x,a) \cup A(a,y)} K^p \frac{e^{-pd(a,z)}}{{f(d(a,z))}}d\mu(z),\end{align*}
where $K=\alpha e^{-\eps C} - \beta e^{-\eps (2D+2C)} \geq e^{-\eps C}(\alpha-\beta e^{-C\eps})>0$.

Assume now that $d(x,y) \geq 4C'+2C$ and take $z \in [a,x]$ such that $d(a,z) = C'+\frac{C}{2}$ and $d(z,x) \geq C'+\frac{C}{2}$ (such an element exists because $X$ is geodesic), one checks that $B(z,\frac{C}{2}) \subseteq A(x,a)$. Moreover, for any $z' \in B(z,\frac{C}{2})$, we have $d(a,z') \leq d(a,z)+\frac{C}{2}=C'+C$. Hence, for any $a \in A(x,y)$ and $z$ as above, we have
\begin{align*} \|\vec{ax} - \vec{ay}\|^p &\geq \int_{z' \in B(z,\frac{C}{2})} K^p \frac{e^{-pd(a,z')}}{{f(d(a,z'))}}\geq \int_{z' \in B(z,\frac{C}{2})} K^p \frac{e^{-p(C'+C)}}{{f(C'+C)}}\\
&\geq \mu(B(z,\frac{C}{2})) K^p \frac{e^{-p(C'+C)}}{{f(C'+C)}}\geq v K^p \frac{e^{-p(C'+C)}}{{f(C'+C)}}=K'>0.\end{align*}
where $v>0$ is a constant given by the non-collapsing assumption on the measure $\mu$. Hence, for any $x,y \in X$ such that $d(x,y) \geq 4C'+2C$ and for any $p \geq 1$, we have
$$\int_{a\in X} \|\vec {a x}-\vec{ay}\|^p d\mu(a) \geq \int_{a \in A(x,y)} K' d\mu(a) = \mu(A(x,y)) K'.$$
To conclude that the tangent bundle is proper we only need to see that $\mu(A(x,y))$ grows at least like the distance between $x$ and $y$. To do that, fix $x',y' \in X$ on a geodesic from $x$ to $y$ such that $d(x,x')=C'+\frac{C}{2}$ and $d(y,y')=C'+\frac{C}{2}$. Since the $\frac{C}{2}$-neighbourhood of a geodesic $[x',y']$ is contained in $A(x,y)$, we have
$$ \mu(A(x,y)) \geq v \left( \frac{d(x',y')}{C}-1 \right) \geq \frac{v}{C}(d(x,y)-2C'-2C).$$ 
As a consequence, we have for any $x,y \in X$ with $d(x,y) \geq 4C'+2C$
$$\int_{a\in X} \|\vec {a x}-\vec{ay}\|^p \geq \frac{K'v}{C}(d(x,y)-2C'-2C),$$
so the tangent bundle is proper.
\end{proof}%

%%%%%%%%%%
\section{Applications}
%%%%%%%%%%%%
We can now finish the proof of our main result.
\begin{proof}[Proof of Theorem \ref{main}]%With our assumptions the critical exponent is the same as the volume entropy of an orbit (with the counting measure and the induced metric), hence we conclude
  Consider $X, G, \mu$ and $p$ as in the statement. By  Proposition \ref{prop:TX}, $X$ admits a $G$-equivariant $p$-uniform tangent bundle of curvature at most $\log(2)/\delta$.  By Theorem \ref{thm:properaction}, $G$ admits a proper action on $L^p(X, TX)$, and with Remark \ref{rem;Fubini}, we get the desired action.
  %together with Proposition \ref{prop:TX}.
\end{proof}
\begin{proof}[Proof of Corollary \ref{propre}]
Let $n$ be the dimension of the manifold $M$ and we can assume, by rescaling,  that the sectional curvature of $M$ lies in the interval $-1 \leq -\alpha$. We will use comparisons with the hyperbolic space, with its Riemannian structure $(\mathbb{H}^n, g_{\mathbb{H}^n})$.  Let $\mu$ be the volume form on $M$. It is non collapsing because the curvature is less than $-\alpha$ (by Gunther's inequality \cite[Theorem 3.101 (ii)]{GHL}, as already mentionned).  Since the curvature is bounded from below, we can apply Bishop's inequality \cite[Theorem 3.101 (i)]{GHL} to obtain that the volume entropy $h$ of the measure $\mu$ is at most that of the $n$-dimensional real hyperbolic space, i.e. $h \leq n-1$. Furthermore, the manifold $M$ has hyperbolicity constant that  that of the $n$-dimensional real hyperbolic space rescaled by $\frac{1}{\sqrt{\alpha}}$, which is $\frac{\log(2)}{\sqrt{\alpha}}$.

Therefore we can apply Theorem \ref{main}, to conclude that for any $p > \frac{n-1}{\sqrt{\alpha}}$, $G$ acts properly by affine isometries on a $L^p$ space. \end{proof}
To finish our proof of Cornulier-Tessera-Valette's result in \cite{CTV}, recall that the \emph{critical exponent} of an action of a group $G$ on a metric space $(X,d)$ is defined as 
$$\delta_c:=\limsup_n \frac{\log N(x,r)}{r}$$
where $N(x,r)$ is the cardinal of $\{\gamma\in G\mid d(x,\gamma x)\leq r\}$.

\begin{proof}[Proof of Corollary \ref{CTV}]
The Riemannian metric on a rank one symmetric space $X$ for $G$,  is CAT(-1) and defines the Hausdorff dimension $\delta_H$ of the visual boundary. We consider $\mu$ the volume form of this Riemannian manifold. Since $G$ acts transitively and measure preserving, $\mu$ is non-collapsing.   According to~\cite[Theorem~0.2]{paulin_critical_exponent}, the critical exponent of any cocompact lattice in $X$ is equal to $\delta_H$, hence the volume entropy of $X$ is equal to $\delta_H$. Let $TX$ be the Riemannian tangent bundle, which has curvature $\leq -1$ (see Example \ref{ex;1.2}). Note that furthermore $TX$ is Borel-isomorphic to $X\times \mathbb R^n$ where $n=\dim(X)$ and hence it is $p$-uniform for all $p$. Recall that $X$ is Gromov $\log(2)$-hyperbolic. Hence 
according to Theorem \ref{thm:properaction}, we deduce that for any $p > \max(1,\delta_H)$, $G$ has a proper affine action on $L^p(X,TX)$. Using Remark \ref{rem;Fubini}, we get the desired action.

%Note that furthermore $TX$ is Borel-isomorphic to $X\times \mathbb R^n$ where $n=\dim(X)$ and hence it is $p$-uniform, so that we can write $L^p(X,TX)$ as some $L^p(\Omega)$.
\end{proof}

\bibliographystyle{alpha}
\bibliography{biblio}

\end{document}